\documentclass{amsart}

\usepackage[parfill]{parskip}  
\usepackage{amssymb}
\usepackage{mathtools}
\newtheorem {note}{Note} 
\newtheorem{state}{Statement}

\newtheorem{lemma}{Lemma}
\newtheorem{step}{Step}

\newtheorem {cor}{Corollary}

\begin{document}
\title{Definability lattice for addition of rationals}
\author[A.L. Semenov, S.F. Soprunov]{A.L. Semenov, S.F. Soprunov}
\begin{abstract} 
In the present paper we discuss the lattice of reducts of  $\langle \mathbb{Q}, \{$+$\}\rangle $.
\end{abstract}
\maketitle
\section{Preliminaries }

We consider the structure $\mathcal{M}=\langle \mathbb{Q}^{< \omega},\{+\} \rangle$, where $\mathbb{Q}^{< \omega} \subset \mathbb{Q}^\mathbb{N}, \vec v \in \mathbb{Q}^{<\omega}$ if $\{i | v_i \ne 0\}$ is finite. We denote by $\vec 0$ the vector $\langle 0,\dots,0,\dots \rangle \in \mathbb{Q}^{<\omega}$. It's well known that $\mathcal{M}$ is $\omega$-saturated elementary extension of $\langle \mathbb{Q}, \{+\}\rangle $, so the lattice of definable reducts of $\mathcal{M}$ (and $\langle \mathbb{Q}, \{+\}\rangle $) corresponds to the lattice of subgroups of the group of permutations $Sym(\mathcal{M})$, containing the group $GL(\mathcal{M})$ of invertible linear maps.

Our consideration consists from 3 parts.   

\textbf{Dyadic relations.} Here we discuss dyadic relations and 2-definable relations -- relations, definable by the signature $\{y=r x | r \in \mathbb{Q}\}$. These relations are almost trivial, but form rather complicated infinite lattice.

\textbf{Triadic relations.} Roughly speaking triadic relations add 2 new reducts: $z=(x+y)/2$ and  $z=\pm x \pm y$ (by $y=\pm x_1 \dots \pm x_n$ we denote the relation $\bigvee_{s \in \{-1,1\}^n}y=\sum_{i=1}^n s(i)x_i$). In particular we reprove, that the group $AGL(\mathcal{M})$ is maximal (\cite{kap}).
 
\textbf{Relations with more then 3 arguments.} We show that they add no new reducts.

If a relation $R(x_1,\dots,x_n)$ is definable in $\mathcal{M}$, tuples $a_1,\dots,a_n$ and $b_1,\dots,b_n$ are linearly independent, then  $R(\bar a) \equiv R(\bar b)$, so we will suppose, that $\lnot R(\bar a)$ for a linearly independent tuples $\bar a$. In other words we suppose that $\{\sum_{i=1}^n r_ix_i \ne \vec 0 | r_i \in \mathbb{Q}, r_i \ne 0$ for some $i\} \cup \{R(x_1,\dots,x_n)\}$ is inconsistent, so 
\begin{note}\label{note1}
$R(x_1,\dots,x_n) \to \bigvee_{j=1}^K \sum_{i=1}^n r_{j,i} x_i = \vec 0$ for some $K$.
\end{note}

From now on a definable relation is a relation, definable in the structure $\langle  \mathcal{M}, \{+\}\rangle $, definable by a signature $\Sigma$ means definable in the structure $\langle  \mathcal{M}, \Sigma\rangle $; independent tuple $\bar t$ is a linearly independent $\bar t=\langle t_1,\dots,t_n\rangle, t_i \in \mathcal{M}$; by $l(a,b), a,b \in \mathcal{M}$ we denote the straight line passing through $a$ and $b$.

We say, that a tuple $a_1,\dots,a_n$ is \emph{$m$-independent} for some natural  $m$ if $\sum_{i=1}^n (k_i/l_i)a_i = \vec 0, |k_i|<m, |l_i|<m$ implies all $k_i=0$.

Due to standard compactness arguments we note, that

\begin{note}\label{note2}
For any definable relation $R(\bar x)$ exists such natural number $m$ that $R(\bar x)$ holds for any independent $\bar x$ iff
$R(\bar x)$ holds for any $m$-independent $\bar x$
\end{note}

We use abbreviations $(\exists_{>k} x)Q(x)$ for $(\exists x_1,\dots,x_{k+1})(\bigwedge_{i \ne j} x_i \ne x_j \land \bigwedge_{i=1}^n Q(x_i))$, $(\exists_{<k} x)Q(x)$ for $(\exists x) Q(x) \land \lnot (\exists_{>k-1} x)Q(x)$, and $(\exists_{=k} x)Q(x)$ for $(\exists_{>k-1} x)Q(x) \land \lnot (\exists_{>k} x)Q(x)$. Sometimes we use the abbreviation $(\exists_{>k} x_1,\dots,x_n)$ which is defined by induction: $(\exists_{>k} x_1,\dots,x_n)Q(x_1,\dots,x_n) \rightleftharpoons (\exists_{>k} x_1)((\exists_{>k} x_2,\dots,x_n)Q(x_1,\dots,x_n))$.

\begin{note} \label{note2-1}
For a definable relation $R(\bar x,\bar y)$ there exists a natural number $K$ such that holds
$$\lnot(\exists_{>K} \bar y)R(\bar x,\bar y) \lor \lnot(\exists_{>K} \bar y)\lnot R(\bar x,\bar y)$$
\end{note}
\begin{proof}
This is shown by contradiction. Suppose that $(\exists x)((\exists_{>K} \bar y)R(\bar x,\bar y) \land (\exists_{>K} \bar y)\lnot R(\bar x,\bar y))$ holds for any $K$. Then the set $\{R(\bar a,\bar b),\lnot R(\bar a,\bar b')\} \cup \{\sum_i p_i a_i + \sum q_i b_i \ne \vec 0 | p_i, q_i \in \mathbb Q, q_i \ne 0 \text{ for some } i\} \cup \{\sum_i p_i a_i + \sum q_i b'_i \ne \vec 0 | p_i, q_i \in \mathbb Q, q_i \ne 0 \text{ for some } i\}$ is consistent, hence due to $\omega$-saturation of $\mathcal{M}$ there are $\bar a, \bar b, \bar b' \in \mathcal M$ such that $\bar b, \bar b'$ are independent, $\mathcal{V}(\bar a) \cap \mathcal{V}(\bar b) = \mathcal{V}(\bar a) \cap \mathcal{V}(\bar b') = \{\vec 0\}$, and $R(\bar a,\bar b), \lnot R(\bar a,\bar b')$. Contradiction, because there is $\sigma \in GL(\mathcal M)$ such that $\sigma(\bar a)= \bar a, \sigma(\bar b)= \bar b'$.
\end{proof}

So the note \ref{note2} can be reformulated as

\begin{cor} \label{note3}
For any definable relation $R(\bar x, \bar y)$ exists a natural number $K$ such that $R(\bar a, \bar b)$ holds for some (any) independent $\bar b$ such that $\mathcal{V}(\bar a) \cap \mathcal{V}(\bar b) = \{\vec 0\}$ iff
$(\exists_{>K} \bar y)R(\bar a,\bar y)$ holds.
\end{cor}

Note, that there is only one nontrivial definable subset of $\mathcal{M}$, i.e. $\{\vec 0\}$.

\section{Dyadic relations}

\begin{state}\label{st0}
If $R(x_1,\dots,x_n)$ is nontrivial 2-definable relation, then $\{\vec 0\}$ is definable by $R$. 
\end{state}
\begin{proof}

We may suppose that $R(a_1,\dots,a_n) \Rightarrow a_i \ne a_j$.

The \emph{rank} of a relation $R(x_1,\dots,x_n)$ is a maximum number $m$, such that $R(a_1,\dots,a_n)$ holds for a tuple $\bar a$, containing $m$ independent items. Note, that we consider relations  $R(x_1,\dots,x_n)$ which rank is less than $n$.

Let $m$ be the rank of $R$. Then (renumbering variables if necessary) $R(a_1,\dots,a_m,b_1,\dots,b_{n-m})$ holds for the independent tuple $\bar a$ and some tuple $\bar b$. The relation $R$ is 2-definable and $m$ is the rank of $R$, so (1) each $b_j \in l(\vec 0, a_i)$ for some $i$ and (2) $\{\bar b' | R(\bar a, \bar b')\}$ is finite (Note \ref{note2-1}). We suppose, that $b_1 \in l(\vec 0, a_1)$ hence for some $k \ne 0$ and sufficiently large $M$ holds $(\exists_{=k} y_1)(\exists_{>M} x_2,\dots,x_m)(\exists y_2,\dots,y_{n-m})(y_1 \ne a_1 \land R(a_1,x_2,\dots,x_m,y_1,\dots, y_{n-m}))$. Denote by $Q(x,y)$ the statement $(\exists_{>M} x_2,\dots,x_m)(\exists y_2,\dots,y_{n-m}) R(x,x_2,\dots,x_m,y,\dots, y_{n-m})$. We see that $|\{c | Q(a_1,c), c \ne a_1\}| = k$. Therefore $|\{c | Q(d,c), c \ne d\}| = k$ for any $d \ne \vec 0$. But $\{c | P(\vec 0,  c), c \ne \vec 0\}$ is or empty or infinite for any definable $P$.  
\end{proof}

Let $R(x,y)$ is definable relation, we may suppose that $(\exists x)(\exists y)(R(x,y) \land x \ne y \land x \ne \vec 0 \land y \ne \vec 0)$ -- otherwise it's equivalent to a definable subset of $\mathcal{M}$.

According to the note \ref{note1} a $R(x,y) \land x \ne \vec 0$ is equivalent to $\bigvee_{i=1}^n y=r_i x$ for some $r_1,\dots,r_n,  r_i \ne 1,0$. 


Denote by  $G$ the multiplicative group generated by set $\{r_1,\dots,r_n\}$, and define the equivalence relation $\sim$ on $\mathcal{M}$ such, that $a \sim b \leftrightharpoons a=rb$ for some $r \in G$. A permutation $\varphi \colon \mathcal{M} \to \mathcal{M}$, preserving the relation $R$, is a composition of a permutation on $\mathcal{M} / \sim$ and bijections between corresponding classes of the equivalence. 

For each $a \in \mathcal{M}$ there is a corresponding permutation  $\sigma_a$ on the set $\{r_1,\dots,r_n\}$, such that $\varphi(a \cdot r_i)= \varphi(a) \cdot \sigma_a(r_i)$. These permutations $\sigma_a$ describe the corresponding bijections. 

If $a \nsim b$, then permutations $\sigma_a$ and $\sigma_b$ are independent. If $a \sim b$, then permutations $\sigma_a$ and $\sigma_b$ are nearly the same.


\begin{state}\label{st1}
If $a \sim b$, then  $|\sigma_a(r_i)|=|\sigma_b(r_i)|$. 
\end{state}
\begin{proof}
It's enough to show, that for any $a \in \mathcal{M}, r_i,r_j$ holds $|\sigma_a(r_j)| = |\sigma_{a \cdot r_i}(r_j)|$. We enumerate $r_1,\dots,r_n$ such, that $|\sigma_a(r_1)| \leqslant |\sigma_a(r_2)| \leqslant \dots,|\sigma_a(r_n)|$. 

By induction on $m \leqslant n$ we prove that
\begin{equation*}
|\sigma_{a \cdot r_i}(r_j)|=|\sigma_a (r_j)|; |\sigma_{a \cdot r_j}(r_i)|=|\sigma_a (r_i)| \mbox{ for all } i \leqslant m, j \leqslant n
\end{equation*}

For a current $m$ we need to prove that $|\sigma_{a \cdot r_m}(r_j)|=|\sigma_a (r_j)|$ and $|\sigma_{a \cdot r_j}(r_m)=\sigma_a (r_m)|$ for all $j \geqslant m$.

The proof is by induction on $j$. 

First we show, that $|\sigma_{a \cdot r_m}(r_j)|=|\sigma_a (r_j)|$. Suppose not, so $\sigma_{a \cdot r_m}(r_i)=\sigma_a (r_j)$ for some $r_i ,|\sigma_a(r_i)|\ne |\sigma_a(r_j)|$. Then $i > j$, because if $i < j$ then, by the induction hypothesis, $|\sigma_{a \cdot r_m}(r_i)|=|\sigma_a (r_i)|$. Then $\varphi (a) \cdot \sigma_a (r_m) \cdot \sigma_a (r_j) = \varphi (a \cdot r_m \cdot r_i)= \varphi (a) \cdot \sigma_a (r_i) \cdot \sigma_{a \cdot r_i} (r_m)$, i.e. $\sigma_a (r_m) \cdot \sigma_a (r_j)= \sigma_a (r_i) \cdot \sigma_{a \cdot r_i} (r_m)$. Because $|\sigma_a (r_i)| > |\sigma_a (r_j)|$, then $|\sigma_{a \cdot r_i} (r_m)| < |\sigma_a (r_m)|$, and $|\sigma_{a \cdot r_i} (r_m)| = |\sigma_a (r_k)|$ for some $k < m$. But according the induction hypothesis $|\sigma_{a \cdot r_i} (r_k)| = |\sigma_a (r_k)|$ holds for all $k < m, i <n$. Contradiction.

Show now that $|\sigma_{a \cdot r_j}(r_m)|=|\sigma_a (r_m)|$. Note that $\varphi (a \cdot r_m \cdot r_i)=\varphi (a) \cdot \sigma_a (r_m) \cdot \sigma_{a \cdot r_m} (r_j) = \varphi (a) \cdot \sigma_a (r_j) \cdot \sigma_{a \cdot r_j} (r_m)$. Because it was already shown that $|\sigma_{a \cdot r_m}(r_j)|=|\sigma_a (r_j)|$, we conclude, that $|\sigma_{a \cdot r_j}(r_m)=\sigma_a (r_m)|$.

End of induction on $j$.

End of induction on $m$.
\end{proof}
\begin{state}\label{st2}
If $|r_i|=|r_j|$ then $|\sigma_a (r_i)|=|\sigma_a (r_j)|$
\end{state}
\begin{proof}
$|\varphi (a \cdot r_i \cdot r_i)|=|\varphi (a)| \cdot |\sigma_a (r_i)| \cdot |\sigma_{a \cdot r_i} (r_i)|=|\varphi (a)| \cdot |\sigma_a (r_i)| \cdot |\sigma_a(r_i)|=|\varphi (a \cdot r_j \cdot r_j)|$
\end{proof}

\subsection{Dyadic relations summary.} 

Automorphism groups of definable dyadic relations closely connected with automorphism groups of finitely generated abelian groups (e.g. \cite{kos}).

First of all we describe a group $G_R$ of automorphisms for a relation $R(x,y) \equiv \bigvee_{i=0}^{n-1} x=r_i y$ where $|r_i| \ne |r_j|$. Let $G$ be the multiplicative group generated by set $\{r_1,\dots,r_n\}$. We say that permutation $\sigma$ on the set $\{r_1,\dots,r_n\}$ is \emph{correct} if the mapping $\psi(r_1^{k_1} \dots r_n^{k_n}) = \sigma(r_1)^{k_1} \dots \sigma(r_n)^{k_n}$ is an automorphism of $G$, in other words if $\sigma$ preserves all  multiplicative dependences between $\{r_1,\dots,r_n\}$.

A permutation $\varphi$, preserving the relation  $R$ on the structure $\mathcal{M}$ is the composition of a permutation on $\mathcal{M} / \sim$ and the permutations $\varphi_f$ of group $G$ for each bijection $f$ between corresponding classes of the equivalence. Each permutation $\varphi_f$ corresponds to an automorphism generated by a correct permutation on the set $\{r_1,\dots,r_n\}$.

A relation $R'$ is definable by a relation $R$ if $G_R \subset G_R'$, i.e. if all $r'_i$ belongs to the group, generated by the set $\{r_1,\dots,r_n\}$ and each correct permutation on $\{r_1,\dots,r_n\}$ generates (correct) permutation on  $\{r'_1,\dots,r'_{n'}\}$. 

A group $G_R$ is a bit more complicated when $|r_i|=|r_j|$ for some $i,j$. Denote by $P=\{s_1,\dots,s_k\}$ subset of such numbers from $\{r_1,\dots,r_n\}$ that $-s_i \not \in \{r_1,\dots,r_k\}$ and by $G'$ the multiplicative group generated by $P$. Let $\phi$ be a mapping of cosets of $G'$ in $G$ to $\{-1,1\}$. Then a permutation $\varphi_f$ of group $G$, corresponding to a bijection between corresponding classes of the equivalence is $\phi(x)\sigma(x)$ where $\sigma$ is an automorphism, generate by a correct permutation on the set $\{r_1,\dots,r_n\}$ and $\phi$.

We call a relation $R$ \emph{2-(un)definable} if it's (un)definable by dyadic relations, i.e. is (un)definable by the signature $\{y= r x | r \in \mathbb{Q}\}$.

By $LGL(\mathcal{M})$ we denote the group of permutations of $\mathcal{M}$, preserving all 2-definable relations. We note that $LGL(\mathcal{M}) = \langle GL(\mathcal{M}), Sym_l \rangle$, where $Sym_l$ is the group of permutations on the set of straight lines, passing through $\vec 0$.




\section{Triadic relations.}

According the note \ref{note1} holds $R(x,y,z) \to  a_1 x +  b_1 y + c_1 z = \vec 0 \lor \dots \lor a_n x +  b_n y + c_n z = \vec 0$. So there are expressions $a_1 x +  b_1 y + c_1 z = \vec 0,\dots,a_k x +  b_k y + c_k z = \vec 0$ such, that $R(x,y,z) \equiv a_1 x +  b_1 y + c_1 z = \vec 0 \lor \dots \lor a_k x +  b_k y + c_k z = \vec 0$ for any linearly independent $x,y$. We can suppose, due to linearly independence of $x,y$, that $c_i \ne 0$, so we can rewrite the equations in the form $R(x,y,z) \equiv z=p_1 x +  q_1 y \lor \dots \lor z=p_k x +  q_k y$. The list of expressions  $z=p_1 x +  q_1 y,\dots, z=p_k x +  q_k y$ or simply the list of pairs $\langle p_1,q_1 \rangle,\dots,\langle p_k,q_k \rangle$ we will call the \emph{table} of relation $R$.

First of all we simplify the relation $R$.

Consider the relation $R'(x,y,z) \rightleftharpoons R(x,y,z) \land \lnot (\exists_{>K} y')R(x,y',z)$  for sufficiently large natural number $K$ (Note \ref{note2-1}). It's easy to see, that (i) the table of $R'$ is a subset of the table of $R$, and (ii) $q_i \ne 0$ for all $q_i$ from the table of $R'$. So we can remove from the table of $R$ all expressions where $p_i=0$ or $q_i=0$. If we removed all expressions from the table of $R$ it means that $R$ is 2-definable, so from now on we suppose that $p_i \ne 0, q_i \ne 0$ for all expressions in the table. The process of removing lines where $p_i=0$ or $q_i=0$ we'll call \emph{normalization}, the result of normalization is \emph{normal form} of relation.

Now we can suppose that for all independent $x,y$ holds $|\{z | R(x,y,z\}|<k$, so for dependent $x,y$ holds $R(x,y,z) \to z \in l(x,y)$, where $l(x,y)$ is the line  passing through $x,y$. Otherwise we consider the relation $R'(x,y,z) \rightleftharpoons R(x,y,z) \land (\exists_{<k+1}z')R(x,y,z')$.

We call the relation $R(x,y,z)$ \emph{affine} if $p_i+q_i=1$ for all pairs from the table of $R$. In other word $R$ is affine if for any independent $x,y$ holds $R(x,y,z) \Rightarrow  z \in l(x,y)$.

The group  $AGL(\mathcal{M})$ is the group of affine permutations: it contains, beside $GL(\mathcal{M})$, permutations $\sigma(x)=x+v$ for all $v \in \mathcal{M}$.

We are going to prove:

\begin{lemma}\label{main-1}
If $R$ is affine, and permutation $\sigma$ preserves $R$ then $\sigma \in AGL(\mathcal{M})$.
\end{lemma}

\begin{cor}\label{seq1}
If $\vec 0$ is definable by an affine relation $R$, then $R$ is equivalent (as reduct) to $z=x+y$.

If $\vec 0$ is not definable by an affine relation $R$, then $R$ is equivalent (as reduct) to  $z=(x+y)/2$.
\end{cor}

\begin{lemma}\label{main-2}
$\vec 0$ is definable by any nonaffine relation $R$.
\end{lemma}

\begin{state}\label{n_aff}
Nonaffine relation $R$ is
 
(i) equivalent to  $z=x+y$ 

or 

(ii) is equivalent to $\{z=\pm x \pm y, R^*\}$ for some 2-definable relation $R^*$.
\end{state}





\begin{proof} \textbf{lemma \ref{main-1}}.

First we prove

\begin{lemma}\label{main-1-1}
By $R$ can be defined  a relation $S(x,y,z)$, such that 

(i) $\{z | S(x,y,z)\} \subset l(x,y)$ and

(ii) for independent $x,y$ holds $S(x,y, (x+y)/2)$.
\end{lemma}
\begin{proof}\textbf{lemma \ref{main-1-1}}

We say that a relation $S(x,y,z)$ is \emph{r-correct}, if for any independent $x,y$ condition (i) holds and $S(a,b, a+r(b-a))$. The relation  $R$ is $q$-correct for some rational $q, q \ne 0, q \ne 1$. Let us show, that by $r$-correct relation $S(x,y,z)$ can be defined $1/r$ and $r/(1-r)$ correct relations. First, it's easy to see that $S(x,z,y)$ is $1/r$-correct.

Second, note that the relation $(\exists v) (S(x,v,z) \land S(v,x,y))$ is  $r/(1-r)$-correct. Condition (i) follows from the condition (i) for $S$. We need to show, that  if $a,b$ are independent, then $(\exists v) (S(a,v,c) \land S(v,a,b))$ where $c=a+(r/(1-r))(b-a)$. Let  $v=b+(c-a)$. It's easy to check, that $S(a,v,c) \land S(v,a,b)$. 

Using operations $r \to 1/r, r \to r/(1-r)$ we can from any $q$-correct relation build a $1/2$-correct relation. \textbf{End of lemma \ref{main-1-1} proof.} 
\end{proof}

According to \emph{the fundamental theorem of affine geometry}(\cite{sch}) a permutation $\sigma$ of $\mathcal{M}$ belongs to $AGL(\mathcal{M})$ iff $\sigma$ takes any 3 collinear points to 3 collinear points.

So we are going to show,that if $\sigma$ preserves the relation $R$, then it takes any 3 collinear points to 3 collinear points. We start with nonzero points. Suppose, that  $a,b,c$ -- 3 collinear points, $a,b,c \ne \vec 0, c=a+r(b-a), 0<r<1$. Let $S(x,y,z)$ be definable by $R$ $1/2$-correct relation. Because $S(a,b,(a+b)/2)$ holds for any independent $a,b$, it must holds for a $n$-independent $a,b$  when $n$ is sufficiently large. Choose such small vector $\Delta$, that  $b=a+k\Delta, c=a+l\Delta$ for some integers $k,l,k>l$ and pairs $a+i\Delta, a+(i+2)\Delta$ are $n$-independent for $i<k$. Then holds 
\begin{equation}
(\exists x_0,\dots, x_k)(x_0=a \land x_k=b \land x_l=c \land \bigwedge_{i=1}^{k-2} S(x_i, x_{i+2}, x_{i+1}))
\end{equation}
 -- we can set $x_i=a+i\Delta$. Permutation $\sigma$ has to preserve $S$ as well as equation (1). From the condition (i) follows, that $\sigma(x_i)$ has to lie on the same line.
 
So the permutation $\sigma$ takes any 3 collinear nonzero points to 3 collinear points. Show now that the point $\vec 0$ keeps collinearity as well. 
 
To the contrary. For any line $l, \vec 0 \in l$ by the $\sigma'(l)$ we denote the line, containing all points from $\sigma(l \setminus \{\vec 0\})$. Suppose that  $\vec 0$ lies on line $l$ but $\sigma(l)$ does not contain the point $\sigma(\vec 0)$. 

Consider 2 cases. (i) $\sigma(\vec 0) \ne \vec 0$. Take some nonzero point $a$ on the line $l$, consider a line $l'$,  passing through $\sigma(\vec 0), \sigma(a)$. Inverse images of all nonzero points of line $l'$ lie on the line $l$ and inverse images of all nonzero points of line $\sigma'(l)$ lie on the line $l$. Contradiction.

Case (ii). $\sigma(\vec 0) = \vec 0$. Choose another line $l'$ passing through $\vec 0$. Lines $\sigma'(l)$ and $\sigma'(l')$ don't intersect. Take an arbitrary line $l''$, parallel to $l$. Lines $\sigma'(l)$ и $\sigma'(l'')$ are parallel, which contradict the intersection of $\sigma'(l')$ and $\sigma'(l'')$.

\textbf{End of lemma \ref{main-1} proof.}
\end{proof}

\begin{proof}\textbf{Proof of corollary \ref{seq1}}.

Let $R$ be an affine relation. The group of permutation, preserving $R$ is a subgroup of $AGL(\mathcal{M})$. If $\vec 0$ is definable by $R$, then this subgroup preserves $\vec 0$, so it coincides with $GL(\mathcal{M})$. If $\vec 0$ is not definable by $R$, then it contains a shift $x \to x+v$ for nonzero $v$. In this case it coincides with $AGL(\mathcal{M})$.
\end{proof}

To prove the lemma \ref{main-2} we need that $p,q$ satisfy conditions: $p^2+q \ne 0; p+q^2 \ne 0; p \ne q$.

So it may be necessary to transform the relation $R$.
\begin{lemma} \label{add-1}
For any nonaffine relation $R$ there is a relation $R'(x,y,z)$, definable by $R$ which table contains a line $z=p x+q y$, where  $p,q \ne 0; p^2+q \ne 0; p+q^2 \ne 0; p \ne q, p+q \ne 1$.
\end{lemma}

\begin{proof}
If $p=-1, q=-1$ then we consider the relation $Q(x,y,u,z) \leftrightharpoons (\exists v_1,v_2,v_3,v_4)(R(x,y,v_1) \land R(x,u,v_2) \land R(y,u,v_3) \land R(v_1,v_2,v_4) \land R(v_4,v_3,z))$. For independent $x,y,u$ it holds when $v_1=-x-y; v_2=-x-u; v_3=-y-u; v_4=-v_1-v_2; z=-v_4-v_2= -(-(-x-y)-(-x-u))-(-y-u) = -2x$. Take the relation $S(x,z) \leftrightharpoons (\exists_{>M}y)(\exists_{>M}v)Q(x,y,v,z)$ for sufficiently large $M$ (Corollary \ref{note3}). The set $\{z | S(x,z) \}$ is finite and contains $-2x$ for a nonzero $x$. So the table of relation $R'(x,y,z) \equiv (\exists x')(\exists y')(S(x,x') \land S(y,y') \land R(x',y',z))$ contains the item $z=2x+2y$.

If $p=q$, then choose the relation $(\exists z')(R(x,y,z') \land R(z',y,z))$ which contains the line $p^2x+q(p+1)y$.

If $p+q^2 = 0$ or $p^2+q = 0$, then we consider a sequence $R_0 \leftrightharpoons  R, R_{i+1}(x,y,z) \leftrightharpoons  (\exists z_1,z_2)(R_i(x,y,z_1) \land R_i(y,x,z_2)\land R_i(z_1,z_2,z))$ of relations. The table of $R_{i+1}$ contains the line $(p_i^2+q_i^2)x+2p_iq_iy=z$ for any $p_i x + q_i y=z$ from $R_i$. Hence the table of $R_k$ for sufficiently large  $k$ contains a line $p_k x+q_k y=z$ where $p_k,q_k \ne 0; p_k^2+q_k \ne 0; p_k+q_k^2 \ne 0; p_k \ne q_k, p_k+q_k \ne 0$.
\end{proof}

\begin{proof}\textbf{lemma \ref{main-2}}.

Suppose, that $R$ is nonaffine relation, i.e. $p +q \ne 1$ holds for some item $p x+q y=z$ of the table of $R$.

We are going to prove that $\vec 0$ is definable by $R$. To the contrary. 

Due to lemma \ref{add-1} we suppose, that $p,q \ne 0; p^2+q \ne 0; p+q^2 \ne 0; p \ne q$.

We define relations $R_1(x,y,z) \leftrightharpoons  (\exists z')(R(x,y,z') \land R(y, z',z)); R_2(x,y,z) \leftrightharpoons (\exists z')(R(y,x,z') \land R(z', y,z))$. Due to normalization we can suppose that tables $R_1, R_2$ contains no zero items.

We claim, that for some $K>0$, sufficiently large $M$, and for any $b \ne \vec 0$ holds $(\exists_{<K} v, v \ne b)(\exists_{>M} u)(\exists w)(R_1(u,b,w) \land R_2(u,v,w))$. 

First note that the line $p q x + (p + q^2) y$ is in the $R_1$ table, and the line $p q x + (q + p^2) y$ is in the $R_2$ table. 

Take $b \ne \vec 0$, and denote $S= \{ c | c \ne b,  (\exists_{>M} u)(\exists w)(R_1(u,b,w) \land R_2(u,c,w)) \}$ where $M$ is sufficiently large, as in Note \ref{note2-1}.

Show that the set $S$ is nonempty. Choose $c=b (p^2+q)/(q^2+p)$. It's easy to see that $c \ne b$, we will prove that $(\exists_{>M} u)(\exists w)(R_1(u,b,w) \land R_2(u,c,w))$. It's enough to demonstrate that  $(\exists w)(R_1(a,b,w) \land R_2(a,c,w))$ holds for any $a$, which is independent with $b$. For this we can set $w=p q a + (q + p^2) b$. 

The set $S$ can not be too big. Suppose $c \in S$. Because $M$ is sufficiently large, so $(\exists w)(R_1(a,b,w) \land R_2(a,c,w))$ holds for any $a \not \in \mathcal V(\{b,c\})$. It means that $q_2 c = q_1 b$ for some $q_1, q_2, p_0$, such that $p_0 x + q_1 y$ is in the $R_1$ table, and  $p_0 x + q_2 y$ is in the$R_2$ table. Because all table constants are nonzero, there are fixed number of such $c$.

We supposed that $\vec 0$ is indefinable, so $(\exists_{<K} v, v \ne \vec 0)(\exists_{>M} u)(\exists w)(R_1(u,\vec 0,w) \land R_2(u,v,w))$ has to hold. Contradiction.
\textbf{End of lemma \ref{main-2} proof.}
\end{proof}

Due to lemma \ref{main-2} from now we consider $\vec 0$ as the symbol of the signature.

We call a relation $R(x,y,z)$ \emph{simple}, if sentence \textit {$R(a,b,c) \Leftrightarrow$ ($c=pa+qb$ for some table line $px+qy=z$ )} holds not only for independent $a,b$ but also for any nonzero $a,b,c$.

Define $R'(x,y,\Delta,z) \leftrightharpoons  (\exists v)(R(x,\Delta,v) \land R(v,y,z))$, $R''$ is the normal form of $R'(x,x,\Delta,z)$. Now denote by $R^*_M$ the normal form of  $(\exists_{>M} \Delta)(\exists w)(R'(x,y,\Delta,w) \land R''(z,z,\Delta,w))$. The relation $R^*_M$ we call \emph{simplification} of $R$.

\begin{lemma}\label{prim}
For any relation $R$ there is a simple definable by $R$ relation.
\end{lemma}

\begin{proof}
Due to lemma \ref{add-1} we suppose that in the table of $R$ there is a pair such that  $p,q \ne 0; p^2+q \ne 0; p+q^2 \ne 0; p \ne q$.

We show that for sufficiently large $M$ the simplification $R^*_M$ of $R$ is simple.

Note that the table of $R^*_M$ contains lines $(p_1p_2x+q_1y)/(p_3p_4+q_3)$ where $p_ix+q_iy, i=1,2,3,4$ are (not necessarily different) lines of $R$ and $p_1q_2=p_3q_4, p_3p_4+q_3 \ne 0$.

Take $a,b,c \ne \vec 0$. Prove that for sufficiently large $M$ holds $R^*_M(a,b,c) \Leftrightarrow (p^*a+q^*b=c$ \emph{for some} $p^*,q^*$ \emph{from table} $R^*$)

(i) $\Leftarrow$. $c=(p_1p_2a+q_1b)/(p_3p_4+q_3)$ for some 4 lines of $R$, such, that $p_1q_2=p_3q_4$. Choose a vector $\Delta \not \in \mathcal V(\{a,b,c\})$. Then all pairs $\{a, v\} ,\{v, b\}, \{v', c\}$ are independent, $R(a,\Delta,v)$ and $R(c,\Delta,v')$ holds, so $R^*_M(a,b,c)$ holds as well.

(ii) $\Rightarrow$. To the contrary. Suppose that $p^*a+q^*b \ne c$ for all lines of table $R^*_M$, but $(\exists_{>M} \Delta)(\exists w)(R'(a,b,\Delta,w) \land R'(c,c,\Delta,w))$. The $M$ is sufficiently large to ensure that there is a vector $\Delta \not \in \mathcal V(\{a,b,c\})$ and $R'(a,b,\Delta,w) \land R'(c,c,\Delta,w)$ holds (Corollary \ref{note3}). From independency follows that $w=p_1q_2\Delta+p_1p_2a+q_1b=p_3q_4\Delta+(p_3p_4+q_4)c$, where $p_3p_4+q_4 \ne 0$. And again from independency $p_1q_2=p_3q_4, c=(p_1p_2a+q_1b)/(p_3p_4+q_3)$ -- contradiction.
\end{proof}

\begin{lemma}\label{fin}
The relation $z=\pm x \pm y$ is definable by any simple relation.
\end{lemma}

\begin{proof}

Let $R$ be simple relation. The proof of lemma \ref{fin} consists from few steps.

\begin{step}\label{fin1}

Denote be $Q$ the set of second components of the table $\{(p_i,q_i)\}$ of relation $R$.


(i) There is a relation definable by $R$ which table is $P \times Q$ for some nonempty $P$.

(ii) The relation $D(y_1,y_2) \Leftrightarrow y_1=(q_i/q_j)y_2, q_i, q_j \in Q$ is definable by $R$.

\end{step}
\begin{proof}

Let $N$ be the number of items in the table $R$. Define dyadic  relations $W_1(x,z) \leftrightharpoons  (x \ne \vec 0) \land (\exists_{<N}y, y \ne \vec 0)R(x,y,z), W_2(y,z) \leftrightharpoons  (y \ne \vec 0) \land (\exists_{<N}x, x \ne \vec 0)R(x,y,z)$. Note that for independent $\bar a, \bar b$ holds $\lnot W_1(\bar a,\bar b)$ because each table line $p_i \bar a+q_iy=\bar b$ has a nonzero solution in $y$ and all solutions are different. Hence $\{z | W_1(\bar a,z)\}$ is a finite subset of $l(\vec 0,a)$ for $a \ne \vec 0$. From the other hand $W_1(\bar a, p_i\bar a)$ holds for any $a \ne \vec 0$ and $p_i$ from the table $R$ because $p_i \bar a+q_iy=p_i \bar a$ has no nonzero solution.

Consider a simple relation  $R_1(x,y,z) \leftrightharpoons  (\exists v)(W_1(v,x) \land R(v,y,z))$.

Note that table of $R_1$ contains the line $x+q_iy$ for each $q_i \in Q$. Define a relation $R_2(x,y,z) \leftrightharpoons  R_1(x,y,z) \land (\exists_{=K}v)(y \ne v \land R(x,v,z) \land (\exists w)(W_2(w,y) \land W_2(w,v)))$, where  $K$ is the number of items in $Q$. 

We show, that for independent $x,y$ holds $R_2(x,y,z) \Leftrightarrow (p_ix+q_iy=z$ \emph {for $\{p_i, q_i\}$ from the table $R_1$ such that $\{p_i,q\}$ belongs to the table of $R_1$ for any $q \in Q$}).

$\Rightarrow$. If $R_2(x,y,z)$ holds then $p_ix+q_iy=z$ for some $\{p_i,q_i\}$. If $(\exists w)(W_2(w,y) \land W_2(w,v))$ holds then $v=ry$ for some $r \in \mathbb Q$. Hence there are $K$ different rational numbers $r_1,\dots,r_K$ such that $p_jx+q_jr_my=z$ holds for some $\{p_j,q_j\}$. From the independency jf $x,y$ follows that $p_j=p_i$ and all $q_j$ are different.

$\Leftarrow$. Let $p_ix+q_iy=z$ holds for some $p_i$, such that $\{p_i,q\}$ belongs to the table of $R_1$ for any $q \in Q$. Note that $(\exists w) (W_2(w,y) \land W_2(w,v)$ holds for any $v=(q/q_i)y, q \in Q$ so $R_2(x,y,z)$ holds.

The table of $R_2$ is not empty because $R_2(x,y,x+q_iy)$ holds for any $q_i \in Q$.

Define $D(y_1,y_2) \leftrightharpoons  (\exists_{>M} x,z)(R^*(x,y_1,z) \land R^*(x,y_2,z))$ for sufficiently large $M$.
\end{proof}

\begin{step}\label{fin2} 

There is a relation definable by $R$ which table contains lines of form  $a(x-y)=z$ or $a(x \pm y)=z$ only.
\end{step}
\begin{proof}

If $R_2(x,y,z)$ holds for independent $x,y$ then $S_{x,y,z}=\{y' | R_2(x,y',z) (\exists w) (W_2(w,y) \land W_2(w,y')\}$ consists of  items $q y/q_i, q \in Q$, where $z=p_ix+q_iy$. So $S_{x,y,z_1}=S_{x,y,z_2}$ if $z_1=p_1x+q_1y, z_2=p_2x+q_2y$ and or $q_1=q_2$ or $q_1=-q_2$ (in the latter case $Q$ has to contain $-q$ for any $q \in Q$).

Consider the relation  $R_3(x,y,z) \leftrightharpoons  (\exists v)(S_{z,v,x}=S_{z,v,y})$. For independent $a,b$ the equality $S_{x,y,a}=S_{x,y,b}$ holds if  $x,y$ is a solution of system of linear equations $p_ix+q_iy=a, p_jx+q_jy=b$, where $q_i=\pm q_j$. In other words $R_3(x,y,z) \to z=(x-y)/(p_i -p_j)\lor z=(x+y)/(p_i+p_j)$ and in the case $z=(x+y)/(p_i+p_j)$ the table must contain the line $z=(x-y)/(p_i -p_j)$ as well.
\end{proof}

\begin{step}
The relation $z=\pm x \pm y$ is definable by $R$.
\end{step}

Consider the relation  $R_4(x,y,\Delta,z) \rightleftharpoons (\exists v,v')(R_3(z,\Delta,v) \land R_3(x,\Delta,v') \land R_3(v',y,v))$. In fact the  $R_4(x,y,\Delta,z)$ is equivalent to disjunction of equations $z+a_i \Delta=x+a_j \Delta + a_k y$ for nonzero $x,y$ and $\Delta \not \in \mathcal V(\{x,y,z,\})$. Note that the table of $R_4$ is a subset of the table of $R_3$.

We consider two cases.

Case 1. The table of $R_4$ contains lines $z=x+q_iy$ only.

Let $N$ be a number of lines in the table of $R_4$ and denote by $W(y,z) \rightleftharpoons (\exists_{<N}x)(x \ne \vec 0 \land R_4(x,y,z))$. It is easy to note that $W(y,z) \Leftrightarrow z=q_iy$.

Now we define $R_5(x,y,z) \rightleftharpoons (\forall y_1)(W(y_1,y) \Rightarrow R_4(x,y_1,z))$ and show that $R_5(x,y,z) \Leftrightarrow z=x+y$ or $R_5(x,y,z) \Leftrightarrow z=x \pm y$.

The sentence  $R_5(x,y,z)$ holds iff for any $q_i$ there is $q_j$ such that $z=x+(q_i/q_j)y$. It follows that $q_i=\pm q_j$.

Case 2. The table of $R_4$ contains lines $z=x+q_iy$ and lines $z=-x+q_jy$.

According item (ii) Step \ref{fin2} the relation $x=\pm y$ is definable by $R_4$, so we can consider the relation $R_5(x,y,z) \rightleftharpoons R_4(\pm x, \pm y, z)$ the table of which is $z=\pm x \pm q_iy$. Denote by $W(y,z) \rightleftharpoons R_5(z,z,y) \land y \ne \vec 0$. It is easy to see, that $W(y,z) \Leftrightarrow y = 2z / q_i$. We use the same arguments as in case 1 to show that $(\forall y_1)(W(y,y_1) \to R_5(x,y_1,z)) \Leftrightarrow z=\pm x \pm 2 y$. We can now define $z=\pm 2 y, y \ne \vec 0$ as $z \ne \vec 0 \land (\exists_{=2}x)(x \ne \vec 0 \land z=\pm x \pm 2y)$ hence $z=\pm x \pm y$ is definable.
\end{proof}

By the symbol $S^\pm(\mathcal{M})$ we denote the group of permutations of $\mathcal{M}$ which satisfy condition $\sigma(x)=\pm x, x \in \mathcal{M}$. Symbol $GL^\pm(\mathcal{M})$ denotes $\langle GL(\mathcal{M}), S^\pm(\mathcal{M})\rangle $. 

\begin{state} \label{+-}
$GL^\pm(\mathcal{M})$ is the group of automorphisms of the relation $z=\pm x \pm y$
\end{state} 
\begin{proof}
(i) It is obvious that $\sigma \in GL^\pm(\mathcal{M}) \Rightarrow \sigma$ preserves $z=\pm x \pm y$.

(ii) Suppose, that $\sigma$ preserves $z=\pm x \pm y$. Note that $\sigma (p x)=\pm p \sigma(x)$. Choose $v_1,\dots, v_n,\dots$ a basis of $\mathcal{M}$. We can suppose that $\sigma(v_i)=v_i$, so $\sigma(\sum p_i v_i) = \sum \pm p_i v_i$. Consider a sequence $w_k=\sum_{i=1}^k v_i$ and find a such permutation $\sigma' \in GL(\mathcal{M}), \sigma'(v_i)=\pm v_i$, that for any $m$ there is an $n>m, \sigma^*(w_n)=w_n$ where $\sigma^*=\sigma' \circ\sigma$. Take an item $a \in \mathcal{M}$ and show that $\sigma^*(a)=\pm a$. Select such $n$ that $i>n \to (a)_i=0$. Then $\sigma^*(w_n+a)=\sum_{i=1}^n \pm(1+p_i) v_i=\pm \sigma^*(w_n) \pm \sigma^*(a) = \pm w_n \pm \sigma^*(a)$, so $\sigma^*(a)=\pm a$.
\end{proof}

\begin{cor} \label{3-1}
If $R$ is nonaffine relation, then $z=\pm x \pm y$ is definable by $R$. If $GL^\pm(\mathcal{M})$ preserves $R$ then $R$ is equivalent to $z=\pm x \pm y$.
\end{cor}

We postpone the proof of statement \ref{n_aff} till the statement \ref{4-4} where more general situation is discussed.









\section{Relations with more then 3 arguments.}

We prove

\begin{state} \label{s4-1}
If a relation $R(x_1,\dots,x_n)$ is 2-undefinable, then a triadic 2-undefinable relation is definable by  $R$.
\end{state}

The notion of table can be generalized for a relation with more than 3 arguments.

First we prove 

\begin{lemma} \label{l4-2}
If a relation $R(x_1,\dots,x_n)$ is 2-undefinable, then by $R$ can be defined a relation $R'(x_1,\dots,x_k)$ which table contains a line $x_k=\sum p_i x_i$, where $p_1,p_2 \ne 0$.
\end{lemma}
\begin{proof} 

We use the notion \emph{rank} from the statement \ref{st0}.

Proof by induction on number of arguments of $R$, on rank $k$ of $R$ and number of corresponding tuples $x_{i_1}, \dots, x_{i_k}$ of length $k$ . 

Let a rank of $R$ be $k<n$, we can suppose that $R(a_1,\dots,a_k,b_1,\dots,b_{n-k})$ holds for independent $\bar a$. Note that because a rank of $R$ is $k$, there are only finite numbers such tuples $\bar b'$, that $R(\bar a, \bar b')$ holds. 

We consider 2 cases:

(A) There is a tuple $\bar b'$, such that $R(\bar a, \bar b')$ holds and some $b'_j \not \in \bigcup_{i=1}^k l(\vec 0,a_i)$. 

Then it's easy to note that the relation $R'(x_1,\dots,x_k,y_j) \leftrightharpoons  (\exists y_1,\dots,y_{j-1},y_{j+1},\dots,y_{n-k})R(\bar x,\bar y)$ satisfies the required condition.

(B) For any tuple $\bar b$ if $R(\bar a, \bar b)$ holds then each $b_j$ belongs to a line $l(\vec 0,a_i)$. 

In this case we will construct a definable by $R$ 2-definable relation $Q(\bar x,\bar y)$, such that for independent $\bar a$ holds $(\forall \bar y)(R(\bar a,\bar y) \equiv Q(\bar a,\bar y))$.

Next we consider relations $P(x_1,\dots,x_k) \leftrightharpoons  (\forall \bar y)(R(\bar x,\bar y) \equiv Q(\bar x,\bar y)), R_1(\bar x,\bar y) \leftrightharpoons  P(\bar x) \land R(\bar x,\bar y); R_2(\bar x,\bar y) \leftrightharpoons  \lnot P(\bar x) \land R(\bar x,\bar y)$. Note, that $R \equiv R_1 \lor R_2$, so one of them has to be 2-undefinable.

From the definition of $P$ follows that $R_1(\bar x,\bar y) \equiv P(\bar x) \land Q(\bar x,\bar y)$, so if $P$ (with a less number of arguments than $n$) is 2-definable, then$R_1$ is 2-definable as well. From the other hand from properties of $Q$ follows that $P(\bar a)$ holds for independent $\bar a$, so  $(\forall y) \lnot R_2(\bar a, \bar y)$ holds. So the number of independent tuples of length $k$ for $R_2$ is less than for $R$.


So it remains to construct such relation $Q(\bar x,\bar y)$, that

(i) for independent $\bar a$ holds $(\forall \bar y)(R(\bar a,\bar y) \equiv Q(\bar a,\bar y))$.

(ii) $Q$ is 2-definable.

(iii) $Q$ is definable by $R$.

Let us remind  that we consider the case when each  $b_j$ belongs to a line $l(\vec 0,a_i)$. For any tuple $s=\langle m_1,\dots,m_{n-k} \rangle, m_j \leqslant k$ we will define a individual relation  $Q_s$, describing the situation when $b_j \in l(\vec 0,a_{m_j})$ and set $Q \leftrightharpoons \bigvee_s Q_s$.

So we fix a $s=\langle m_1,\dots,m_{n-k} \rangle, m_j \leqslant k$. For each $j \leqslant {n-k}$ define a relation 
$$S_j(x_{m_j},y_j) \leftrightharpoons  (\exists_{>M} x_1,\dots,x_{m_j-1},x_{m_j+1},\dots,x_k)(\exists y_1,\dots,y_{j-1},y_{j+1},\dots,y_{n-k})R(\bar x, \bar y)$$ 
It's clear, that for independent $a \ne \vec 0$ the set $\{y | S_j(a,y)\}$ is finite, and $\{y | S_j(a,y)\} \subset l(\vec 0,a)\}$, and $R(\bar a,\bar b) \to S_j(a_{m_j},b_j)$ for independent $\bar a$ if $b_j \in l(\vec 0,a_{m_j})$. So $S_j(x,y) \equiv y=\alpha_{j,1} x \lor \dots \lor y=\alpha_{j,n_j} x$ for a finite set $\{\alpha_{j,1},\dots, \alpha_{j,n_j}\}$ of rational numbers. Note that $\vec 0$ is definable by $S_j$. Without loss of generality we can suppose that $R(\bar z)$ is false, if at least one argument is equal to $\vec 0$.

We say that a string $t=\langle \alpha^t_1,\dots,\alpha^t_{n-k} \rangle$ of rational numbers is \emph{regular}, if for some (for any) independent  $\bar a$ holds $R(\bar a, \alpha^t_1 \cdot a_{m_1},\dots,\alpha^t_{n-k} \cdot a_{m_{n-k}})$. Define $T_s(\bar x,\bar y) \leftrightharpoons  \bigvee_t y_1=\alpha^t_1 \cdot x_{m_1}\land \dots \land y_{n-k}=\alpha^t_{n-k} \cdot x_{m_{n-k}}$, the disjunction contains all regular strings $t$. Conditions (i) and (ii) for $T_s$ follows from the definition immediately.

Prove now that $T_s$ is definable by $R$. First for each $j \leqslant n-k$ we define an equivalence $E_j$. Without loss of generality we describe the equivalence $E_1$. Let $m_1,\dots, m_l$ be the list of all numbers $m_i \leqslant n-k$ such that $s_{m_i}=1$. We may suppose that they are $1,\dots,l$. We say that 2 tuples $a,b_1,\dots,b_l$ and $a',b'_1,\dots,b'_l$ are equal ($E_1(a,\bar b,a',\bar b')$ holds) if 
\begin{multline*}
\bigwedge_{i \leqslant l} (S_i(a,b_i) \land S_i(a',b'_i))\\
\land (\exists_{>M}x_2,\dots,x_k)(\forall y_{l+1},\dots,y_{n-k})
 (R(a,x_2,\dots,x_k,b_1,\dots,b_l,y_{l+1},\dots,y_{n-k}) \equiv\\
  R(a',x_2,\dots,x_k,b'_1,\dots,b'_l,y_{l+1},\dots,y_{n-k})) 
\end{multline*}
Note 

(*) that if $\{v,u\}$ are independent, then $E_1(v,p_1 \cdot v,\dots,p_l \cdot v, u, p_1 \cdot v,\dots,p_l \cdot v)$ holds for any rationals $p_i \ne 0$.

Show that $T_s(x_1,\dots,x_k,y_1,\dots,y_{n-k})$ is equivalent to
\begin{multline}
(\exists_{>M}x'_1,\dots,x'_k)(\forall y'_1,\dots,y'_{n-k})\\
(\bigwedge_{i=1}^k E_i(x_i,\bar y_i,x'_i, \bar y'_i) \to R(x'_1,\dots,x'_k,y'_1,\dots,y'_{n-k})) 
\end{multline}

Suppose (2) holds. We need to show, that for independent $\bar a$ holds $R(\bar a, \alpha_1 \cdot a_{m_1},\dots,\alpha_{n-k} \cdot a_{m_{n-k}})$, wher $\alpha_i \cdot x_{m_i} = y_i $. It immediately follows from (2) and note (*).

To the opposite. Choose a regular string $\langle \alpha_1,\dots,\alpha_{n-k} \rangle$ and set $y_i=\alpha_i \cdot x_{m_i}$. We need to show that  (2) holds. By the definition of regular string there is an independent tuple $\bar a$, such that $R(\bar a, \bar b)$ holds where $b_i=\alpha_i \cdot a_{m_i}$. Select a tuple $\bar a'$, such that $\bar a \cup \bar a' \cup \bar x$ is independent. Prove that if $E_i(x_{m_i},\bar y_i,a'_{m_i},\bar  b_i)$ holds for all $i$ then $R(\bar a', \bar b')$ holds.  $E_1(a_1,b_1,\dots,b_l,a'_1,b'_1,\dots,b_l)$ holds because $E_1(x_1,y_1,\dots,y_l,a'_1,b'_1,\dots,b'_l)$ and $E_1(x_1,y_1,\dots,y_l,a_1,b_1,\dots,b_l)$. By the definition of $E_1$ follows $R(a'_1,a_2,\dots,a_k,b'_1,\dots, b'_l,b_2,\dots,b_{n-k}) \equiv R(\bar a, \bar b)$. Continuing this procedure we get $R(\bar a',\bar b')$. 
\end{proof}


\begin{lemma} \label{4-2}
If the table of relation $R(x_1,\dots,x_n,z)$ contains a line $\sum_{i=1}^{n} p_i x_i =z$ where $p_1,p_2 \ne 0$ then there is definable by $R$ 2-undefinable triadic relation.
\end{lemma}
\begin{proof}

Due to lemma \ref{add-1} we suppose, that $p_1^2+p_2 \ne 0$.

Consider a relation $R'(x_1,\Delta,x_3,\dots,x_n,z) \rightleftharpoons (\exists v)(R(x_1,\Delta,x_3,\dots,x_n,v) \land R(v,x_2,\dots,x_n,z))$ which contains for independent $\{\Delta,x_1,x_3,\dots,x_n\}$ and $\{\Delta,x_2,\dots,x_n\}$ the line $z=p_1^2 x_1+p_1p_2\Delta+p_2 x_2+p_3(1+p_1)x_3+\dots+p_n(1+p_1)x_n$. So a relation $R^*(x_1,x_2,z) \leftrightharpoons  (\exists_{>M}\Delta,x_3,\dots,x_n)(\exists w)(R'(\Delta,x_1,\dots,x_n,w) \land R'(\Delta,z,z,x_3,\dots,x_n,w))$ holds for independent $\{x_1,x_2\}$ when $(p^2+q)z=p^2 x_1+p_2 x_2$ and there are a finite number of such $z$ that  $R^*(x_1,x_2,z)$ holds.
\end{proof}

So if by 2-undefinable relation $R$ an affine ternary relation is definable, then (according to lemma \ref{main-1}) the automorphism group of $R$ is a subgroup of $AGL(\mathcal{M})$ and coincide with this group if $\vec 0$ is not definable  by $R$, otherwise it coincide with $GL(\mathcal{M})$.

If by 2-undefinable relation $R$ a nonaffine ternary relation is definable, then the relation $z=\pm x \pm y$ is definable as well (lemma \ref{fin}). We are going to prove, that if the group $GL^\pm(\mathcal{M})$ does not preserves $R$ then or $z=x+y$ is definable by$R$ or $R$ is equivalent to signature $\{z=\pm x \pm y, R^*(\bar x)\}$ for some 2-definable relation $R^*$. 

We start with 

\begin{lemma} \label{4-2-1}
Suppose that for a relation $R(\bar x,z)$, a tuple $\bar r \in \mathbb Q, r_i \ne 0$  and for any independent tuple $\bar a \in  \mathcal{M}$ holds

$R(\bar a,z) \Leftrightarrow z=r_1 a_1 \pm r_2 x_2 \dots \pm r_n a_n$

Then $z=x+y$ is definable by $R$.
\end{lemma}
\begin{proof}
Let us remind that the relations $y=-x$ and $y=\pm r x$ for any $r \in \mathbb Q$ are definable by $z=\pm x \pm y$ and hence by $R$.

Define the relation $R_1(x,y,z) \rightleftharpoons (\exists x_2,\dots,x_n)(y=\pm r_2 x_2 \dots \pm r_n x_n \land z=\pm r_1 x \pm y \land R(x, x_2,\dots,x_n,z))$ and note that for independent $x,y$ holds $R_1(x,y,z) \Leftrightarrow z=r_1 x \pm y$. Next define $R_2(x,y,\Delta,z) \rightleftharpoons (\exists u,v) (R_1(z,\Delta,u) \land R_1(y,\Delta,v) \land R_1(x,v,u))$. If $\Delta \not \in \mathcal V(\{x, y, z\})$ then $R_2(x,y,\Delta,z) \Leftrightarrow r_1 z \pm \Delta = r_1 x \pm r_1 y \pm \Delta$ hence $(\exists_{>M} \Delta) R_2(x,y,\Delta,z) \Leftrightarrow z=x \pm y$ for sufficiently large $M$ and $x,y,z \ne 0$. Therefore $z=x+y \Leftrightarrow ((x \ne y \land x \ne \vec 0 \land y \ne 0 \to (z = x \pm y \land z = y \pm x)) \land (x=y \land x \ne \vec 0 \to (z=x \pm x \land z \ne \vec 0)) \land (x= \vec 0 \to z=y) \land (y= \vec 0 \to z=x))$
\end{proof}

\begin{lemma} \label{4-3}

Suppose that for a relation $R(\bar t, x_1,\dots,x_n,z)$ where $n>1$, some $m>0$, any tuple $\bar r \in \mathbb Q, r_i  \ne 0$, any parameters $\bar t \in \mathcal{M}$, and for any $m$-independent tuple $\bar a \in  \mathcal{M}$ holds

$R(\bar t,\bar a,z) \Leftrightarrow \bigvee_{s \in S_{\bar t}} z=\sum_{i=1}^n s(i)r_i a_i $ where $S_{\bar t} \subset \{-1,1\}^n$

and $0 < |S_{\bar t}| < 2^n$ for some $\bar t \in \mathcal{M}$.

Then $z=x+y$ is definable by $\Sigma=\{z=\pm x \pm y, R\}$.
\end{lemma}
\begin{proof}

Note, that $m$-independency is definable by $z=\pm x \pm y$.

We may suppose that $(\forall \bar t)((\exists z) (R(\bar t,\bar a,z) \to |S_{\bar t}| = k))$ holds for any $m$-independent $\bar a$ and some $0 < k < 2^n$, otherwise we consider the relation $R(\bar t,\bar x,z) \land |S_{\bar t}| = k$ for an appropriate $k$.

For any  $m$-independent $\bar x$ and $z$ we denote by $W_{\bar x,z}$ the set $\{z' | (\exists u)(u=\pm r_1 x_1 \dots \pm r_{n-1} x_{n-1} \land z=\pm u \pm r_n x_n \land z'= \pm u \pm r_n x_n)\}$. If $R(\bar t,\bar x, z)$ holds, then
$z' \in W_{\bar x,z} \Leftrightarrow (z=\sum_{i=1}^n s(i)r_i a_i , z'=\pm \sum_{i=1}^n s(i)r_i a_i \pm r_n x_n$ for some $s \in S_{\bar t}$). Note, that $z \in W_{\bar x,z}$ and $|W_{\bar x,z}| = 4$. By $W'_{\bar t, \bar x,z}$ we denote $\{z' | R(\bar t, \bar x,z'), z' \in W_{\bar x,z}\}$.

Now for any $m$-independent $\bar x$ and $z$ such that $R(\bar t, \bar a, z)$ we define the \emph {type} ($T(\bar t, \bar a, z)$) -- a natural number $i, i < 7$ , the relation $T(\bar t, \bar x, z)=i$ will be definable by $\Sigma$ for each $i$.

(i) $T(\bar t, \bar x, z)=1 \rightleftharpoons |W'_{\bar t,\bar x,z}|=4$

(ii) $T(\bar t, \bar x, z)=2 \rightleftharpoons |W'_{\bar t, \bar x,z}|=1$

(iii) $T(\bar t, \bar x, z)=3 \rightleftharpoons |W'_{\bar t, \bar x,z}|=3$

(iv) $T(\bar t, \bar x, z)=4 \rightleftharpoons (|W'_{\bar t, \bar x,z}|=2 \land R(\bar t, \bar x,-z))$

(v) $T(\bar t, \bar x, z)=5 \rightleftharpoons (|W'_{\bar t, \bar x,z}|=2 \land \lnot R(\bar t, \bar x,-z) \land R(\bar t, x_1,\dots,x_{n-1}, -x_n,z))$

(vi) $T(\bar t, \bar x, z)=6 \rightleftharpoons (|W'_{\bar t, \bar x,z}|=2 \land  \lnot R(\bar t, \bar x,-z) \land \lnot R(\bar t, x_1,\dots,x_{n-1}, -x_n,z))$

By $T_R$ we denote the minimal $i$ such that $(\exists \bar t, \bar x, z)((\{x_1,\dots,x_n\} \text{ is m-independent }) \land R(\bar t,\bar x,z) \land T(\bar t,\bar x, z)=i)$ and denote $R_1(\bar t, \bar x,z) \rightleftharpoons R(\bar t,\bar x,z) \land T(\bar t,\bar x, z)=T_R$. Note that for some $S_{1,\bar t} \subset S_{\bar t}, S_{1,\bar t} \ne \varnothing$ and any $m$-independent $\bar a$ holds

$R_1(\bar t,\bar a,z) \Leftrightarrow \bigvee_{s \in S_{1,\bar t}} z=\sum_{i=1}^n s(i)r_i a_i $

For a function $s \in \{-1,1\}^k, \alpha= \pm 1$ we denote by $\langle s, \alpha \rangle \in \{-1,1\}^{k+1}$ extension $s$ on $\{1,\dots,k+1\}$ such that $\langle s, \alpha \rangle(i)=s(i), i < k+1, \langle s, \alpha \rangle(k+1)=\alpha$ by $s^*$ we denote  the initial segment of $s$ i.e. $s=\langle s^*, s(k) \rangle$. If $S \subset \{-1,1\}^k$ then $S^*=\{s^* | s \in S\}$. 

Proof by induction on $n$.

The basis of the induction, case n=2, will be considered later.

For $n>2$ we will consider all 6 cases:

(i) $T_R=1$: if $s \in S_{1,\bar t}^*$ then all 4 strings $\langle \pm s, \pm 1 \rangle$ belong to $S_{1,\bar t}$ as well.
Define the relation $R_2(\bar t,x_1,\dots,x_{n-1},z) \rightleftharpoons z= \sum_{i=1}^{n-1} \pm r_i a_i \land (\exists x_n, z')((\{x_1,\dots,x_n\} \text{ is m-independent }) \land R_1(\bar t, \bar x,z') \land z'=\pm z \pm r_n a_n))$. For $m$-independent $a_1,\dots, a_{n-1}$ holds $R_2(\bar t,\bar a,z) \Leftrightarrow (z= \sum_{i=1}^{n-1} s(i)r_i a_i$ for some $s \in S^*_{1,\bar t}$). Because $|S^*_{1,\bar t}| = |S_{1,\bar t}| / 2$ we can use induction.

(ii) $T_R=2$: if $s \in S_{1,\bar t}^*$ then $-s \not \in S_{1,\bar t}^*$ and $\langle s, 1 \rangle \in S_{1,\bar t} \Leftrightarrow \langle s, -1 \rangle \not \in S_{1,\bar t}$. Consider the relation $R_2(\bar t, \bar x,z,u) \rightleftharpoons R_1(\bar t,\bar x,z) \land u=\sum_{i=1}^{n-1}\pm r_i(1+r_n) x_i \pm r_n^2 x_n \land R_1(\bar t,x_1,\dots,x_{n-1},z,u)$. Let a tuple $\bar a$ be $m$-independent, $u=\sum_{i=1}^{n-1} s_1(i)r_i a_i + s_1(n) r_n (\sum_{i=1}^{n} s_2(i)r_i a_i)$ for some $s_1, s_2 \in S_1$. Then $u=\sum_{i=1}^{n-1}\pm r_i(1+r_n) x_i \pm r_n^2 x_n$ if $s_1^*=s_2^*$ and $s_1(n)=1$ or $s_1^*=-s_2^*$ and $s_1(n)=-1$. From $s \in S_{1,\bar t}^* \to -s \not \in S_{1,\bar t}^*$ follows that $(\exists u)R_2(\bar t,\bar x,z,u) \Rightarrow z=\sum_{i=1}^{n-1} s(i)r_i x_i + r_n x_n$ for $m$-independent $\bar x$. If the relation $(\exists u)R_2(\bar t,x,z,u)$ is nonempty for some parameters $\bar t$ and $m$-independent $\bar x$ then we can use the Lemma \ref{4-2-1} renaming $x_n$ to $x$. If $(\exists u)R_2(\bar t,x,y,z)$ is empty for any parameters $\bar t$ then $s(n)=-1$ for any $s \in S_{1,\bar t}$, hence we can use Lemma \ref{4-2-1} for the relation $R_1$.
 
(iii) $T_R=3$: if $s \in S^*_{1,\bar t}$ then there is only one item $s'=\langle \pm s, \pm 1 \rangle, s' \not \in S_{1,\bar t}$. Define $R_3(\bar t,\bar x,z) \rightleftharpoons (\exists z')(R(\bar t,\bar x,z') \land z \in W_{\bar x,z'} \setminus W'_{\bar t,\bar x,z'})$. The relation $R_3$ meets condition of the case (ii).

(iv) $T_R=4$: $\langle s,1 \rangle \in S_{1,\bar t} \Leftrightarrow \langle -s,-1 \rangle \in S_{1,\bar t}$. Consider the relation $R_2(\bar t,\bar x,z,u)$ from the case (ii). Just as in that case we see, that $(\exists u)R_2(\bar t,\bar x,z,u) \Rightarrow z=\sum_{i=1}^{n-1} \pm r_i x_i + r_n x_n$ for $m$-independent $\bar x$, so we can use the Lemma \ref{4-2-1} for the relation $(\exists u)R_2(\bar t,\bar x,z,u)$.

(v) $T_R=5$: if $s \in S^*_{1,\bar t}$ then $\langle s,1\rangle, \langle s,-1 \rangle \in S_{1,\bar t}, -s \not \in S^*_{1,\bar t}$. Consider the relation $R_3(x_1,\dots,x_{n-1},z,x_n)$ and note that in this case it meets conditions of case (iv).

(vi) $T_R=6$: if $s \in S^*_{1,\bar t}$ then $-s \in S^*_{1,\bar t}$ and $\langle s,1 \rangle \in S^{1,\bar t} \Leftrightarrow \langle -s,1 \rangle \in S_{1,\bar t}$.  Consider the relation $R_2(\bar t,\bar x,z,u)$ from the case (ii). It is easy to check that in this case if $R_2(\bar t,\bar x,z,u)$ holds then $u=\sum_{i=1}^{n-1}\pm r_i(1+r_n) x_i + r_n^2 x_n$ so so we can use the Lemma \ref{4-2-1} for the relation $R_3(\bar t,\bar x,u) \rightleftharpoons (\exists z)R_2(\bar t, \bar x,z,u)$.

Basis of the induction: $n=2$. The same as $n>2$ except the case (i) because $T_R < 4$.

\end{proof}

\begin{state} \label{4-4} 
Suppose that $R$ in 2-undefinable nonaffine relation. Then or $R$ is equivalent to $z=x+y$ or $R$ is equivalent to $\{z=\pm x \pm y, R^*\}$ for some 2-definable relations $R^*$.
\end{state}
\begin{proof}
We know, that $z=\pm x \pm y$ is definable by $R$, so is the relation $y=-x$ and relations $y=\sum_i \pm r_i x_i$ for any $r_i \in \mathbb{Q}$, so the $m$-independency is definable for any $m$ as well.

Let $m$ be such number that $\lnot R(\bar a)$ for any $m$-independent $\bar a$.

We prove by induction of $n$ -- numbers of arguments of $R$.



Let $s \subset \{1,\dots, n\}, s \ne \varnothing, s \ne \{1,\dots, n\}$ and $p=\{r_{i,j} | i \not \in s, j \in s, r_{i,j}=l/k$ for some $|k|,|l|<m \}$. Denote by $R_{s,p}(\bar x)$ the statement ($\{x_i | i \in s\}$ is m-independent $\land \bigwedge_{i \not \in s}  x_i=\sum_{j \in s} \pm r_{i,j} x_j \land R(\bar x)$). The relation $R_{s,p}(\bar x)$ is definable by $R$ and $R(\bar x) \equiv \bigvee_{s,p} R_{s,p}(\bar x)$. 

We prove the statement for each $R_{s,p}(\bar x)$, without loss of generality we suppose that $s=\{1,\dots,k\}$.

Consider different cases:

(i) For any $i > k$ there is only one $l_i \leqslant k$ such that $r_{i,l_i} \ne 0$. In other words $R_{s,p}(\bar x)$ is equivalent to ($\{x_1,\dots,x_k\}$ is m-independent $\land \bigwedge_{i > k}  x_i= \pm r_{l_i} x_{l_i} \land R(\bar x)$). Then $R_{s,p}(\bar x)$ is equivalent to ($\{x_1,\dots,x_k\}$ is m-independent $\land \bigvee_{\bar p \in A} \bigwedge_{i > k}  x_i= p_i  x_{l_i}$) where $A=\{ \bar p | p_i= \pm r_{l_i}, R(a_1,\dots,a_k, p_{k+1, l_{k+1}} x_{l_{k+1}},\dots, p_{n, l_{n}} x_{l_{n}})$ holds for independent $\{a_1,\dots,a_k\}\}$.

(ii) $r_{n,1},\dots,r_{n,m} \ne 0,r_{n,m+1},\dots,r_{n,k} = 0$, where $1<m \leqslant k$.

Consider the relation $R_{s,p}(x_1,\dots,x_k,t_1,\dots,t_{n-k-1},z)$. If for any $m$-independent $\{a_1,\dots,a_k\}$ holds $(\exists \bar t)(0 < |\{z | R_{s,p}(a_1,\dots,a_k,t_1,\dots,t_{n-k-1},z)| < 2^m)$ then, due to Lemma \ref{4-3}, the relation $z=x+y$ is definable by $R_{s,p}$ and hence by $R$.

If for any $m$-independent $\{a_1,\dots,a_k\}$ holds 
$$(\forall \bar t)((\exists z) R_{s,p}(a_1,\dots,a_k,t_1,\dots,t_{n-k-1},z) \to |z | R_{s,p}(a_1,\dots,a_k,t_1,\dots,t_{n-k-1},z)| = 2^m)$$
then 
$$R_{s,p}(x_1,\dots,x_k,t_1,\dots,t_{n-k-1},z) \equiv (\exists z) R_{s,p}(x_1,\dots,x_k,t_1,\dots,t_{n-k-1},z) \land z=\sum_{i=1}^m \pm r_{n,i} x_k$$
so we can use induction for the relation $(\exists z) R_{s,p}(x_1,\dots,x_k,t_1,\dots,t_{n-k-1},z)$
\end{proof}

The group of permutations, preserving all relations of form $\{z=\pm x \pm y, y=p_1 x,\dots,y=p_n x\}$ we denote $LGL^\pm(\mathcal{M})$: beside $GL(\mathcal{M})$ it contains such permutations $\sigma$ that for each line $l$ passing through $\vec 0$ or $\sigma(x)=x, x \in l$ or $\sigma(x)=-x, x \in l$.

\section{Summary.} 


\includegraphics[scale=0.5]{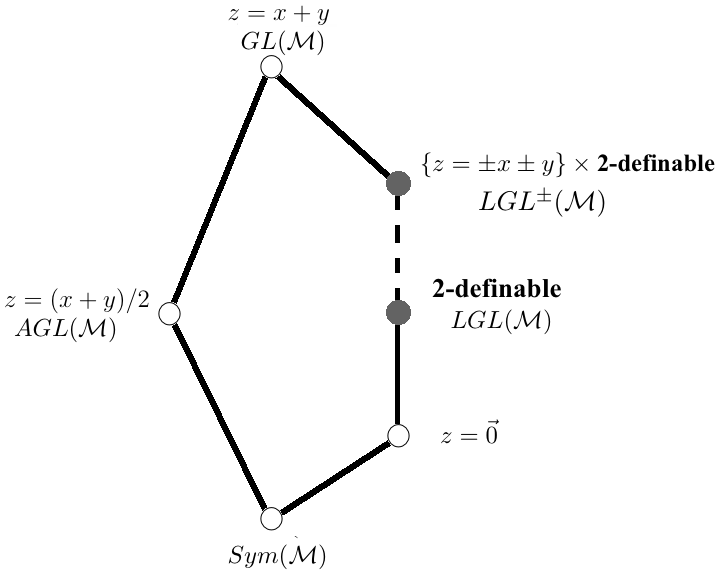}

\emph{Comments:}

(i) Solid gray vertex denotes a family of relations

(ii) Dotted line means that members of one family is defined by members of another one.

\emph{Acknowledgement}. We would like to thank Fedor Yakovlev for useful discussions and Albert Muchnik for the interest to the subject.


\end{document}